\documentclass[brochure,12pt,english]{bourbaki}
\usepackage[matrix,arrow]{xy}

\newtheorem{conv}[defi]{Convention}
\newcommand{\congpf}{\stackrel{\sim}{\to}}
\newcommand{\Diff}{{\rm Diff}}

\newcommand{\Teich}{{\rm Teich}}

\newcommand{\Def}{{\rm Def}}

\newcommand{\Pic}{{\rm Pic}}
\newcommand{\Hilb}{{\rm Hilb}}
\newcommand{\gm}{{\mathfrak M}}

\newcommand{\Hom}{{\rm Hom}}
\newcommand{\id}{{\rm id}}

\newcommand{\cal}{\mathcal}

\newcommand{\kc}{{\cal C}}

\newcommand{\kk}{{\cal K}}

\newcommand{\ko}{{\cal O}}
\newcommand{\kp}{{\cal P}}

\newcommand{\kt}{{\cal T}}
\newcommand{\kx}{{\cal X}}
\newcommand{\ky}{{\cal Y}}

\newcommand{\ZZ}{\mathbb{Z}}

\newcommand{\RR}{\mathbb{R}}
\newcommand{\CC}{\mathbb{C}}

\newcommand{\II}{\mathbb{I}}

\newcommand{\PP}{\mathbb{P}}
\newcommand{\qqed}{\hspace*{\fill}$\Box$}

\date{Juin 2011}
\bbkannee{63\`eme ann\'ee, 2010-2011}
\bbknumero{1040}
\title{A Global Torelli theorem\\ for hyperk\"ahler manifolds}
\subtitle{after M.\ Verbitsky}
\author{Daniel HUYBRECHTS}
\address{Universit\"at Bonn\\
Mathematisches Institut\\
Endenicher Allee 60\\
D--53115 Bonn, Germany}
\email{huybrech@math.uni-bonn.de}
\begin{document}
\maketitle

Compact hyperk\"ahler manifolds are  higher-dimensional generalizations of K3 surfaces. The
classical Global Torelli theorem for K3 surfaces, however, does not hold in  higher dimensions.
More precisely, a compact hyperk\"ahler manifold  is in general
not determined by its natural weight-two Hodge structure.
The text gives an account of a recent theorem of M.\ Verbitsky, which can be regarded
as a weaker version of the Global Torelli theorem  phrased in terms of the injectivity of the period
map on the connected components of the moduli space of marked manifolds.

\section{Introduction}

The Global Torelli theorem is said to hold for a
particular class of  compact complex algebraic or K\"ahler  manifolds
if any two manifolds of the given type can be distinguished by
their integral Hodge structures.

The most prominent examples for which a Global Torelli theorem has been proved classically
include  complex tori and complex curves.
Two complex tori $T=\CC^n/\Gamma$ and $T'=\CC^n/\Gamma'$
are biholomorphic complex manifolds if and only if there exists an isomorphism of weight-one Hodge structures
$H^1(T,\ZZ)\cong H^1(T',\ZZ)$. Similarly, two smooth compact complex curves $C$ and $C'$ are isomorphic
if and only if there exists an isomorphism of weight-one Hodge structures $H^1(C, \ZZ)\cong H^1(C',\ZZ)$
that in addition respects the intersection \mbox{pairing.}

\smallskip

Here, we are interested in  higher-dimensional analogues of
the following Global Torelli theorem for K3 surfaces. 

\smallskip

$\bullet$ {\it Two complex K3 surfaces $S$ and $S'$ are isomorphic if and only if
there exists an isomorphism of Hodge structures $H^2(S,\ZZ)\cong H^2(S',\ZZ)$ respecting
the intersection \mbox{pairing.}}

\smallskip

The result is  originally due to  Pjatecki{\u\i}-{\v{S}}apiro and {\v{S}}afarevi{\v{c}} 
in the algebraic case and to Burns and Rapoport for K3 surfaces of K\"ahler type (but as Siu proved later,
every K3 surface is in fact K\"ahler), see \cite{BeauvAst,BeauvK3} for details and references.

\smallskip

Recall that for complex tori $T$ and $T'$, any Hodge isomorphism $H^1(T,\ZZ)\cong H^1(T',\ZZ)$ is induced by an isomorphism
$T\cong T'$.  Also, for complex curves any Hodge isometry can be  lifted up to sign.
A similar stronger form of the Global Torelli theorem holds  for generic K3 surfaces.

\smallskip

$\bullet$ {\it For any Hodge isometry $\varphi:H^2(S,\ZZ)\congpf H^2(S',\ZZ)$ between two generic(!)
K3 surfaces, there exists an isomorphism $g:S\congpf S'$ with $\varphi=\pm g_*$.}

\subsection{Is there a Global Torelli for hyperk\"ahler manifolds?}

Compact hyperk\"ahler manifolds  are the natural higher-dimensional generalizations of K3 surfaces
and it would be most interesting to establish some version  of the  Global Torelli theorem for this important class of
Ricci-flat manifolds. In this context, the second cohomology $H^2(X,\ZZ)$  is the most relevant part
of cohomology  and not the much larger middle cohomology which one would usually consider for arbitrary compact K\"ahler mani\-folds. 
As we will recall in Section \ref{sect:Defi}, 
the second cohomology  of a compact hyperk\"ahler manifold comes with a natural quadratic form, the
Beauville--Bogomolov form, and its canonical weight-two Hodge structure is of a particularly simple type.

So, is a compact hyperk\"ahler manifold $X$  determined up to isomorphism by  its weight-two Hodge structure
$H^2(X,\ZZ)$ endowed with the Beauville--Bogomolov form? 
More precisely, are two compact hyperk\"ahler manifolds $X$ and $X'$ isomorphic if
$H^2(X,\ZZ)$ and $H^2(X',\ZZ)$ are  Hodge isometric, i.e.\ if  there exists an isomorphism
of weight-two Hodge structures $H^2(X,\ZZ)\cong H^2(X',\ZZ)$ that is compatible with
the Beauville--Bogomolov forms on both sides?
Unfortunately, as was discovered very early
on, a Global Torelli theorem for compact hyperk\"ahler manifolds cannot  hold true literally. 

The first counterexample was produced by Debarre in \cite{Deb}: 

\smallskip

$\bullet$ {\it There exist non-isomorphic compact hyperk\"ahler manifolds $X$ and $X'$
with iso\-metric weight-two Hodge structures. }

\smallskip

 See also \cite[Ex.\ 7.2]{Yosh} for examples
with $X$ and $X'$ projective (and in fact isomorphic to certain Hilbert schemes of points on
projective K3 surfaces).

In Debarre's example $X$ and $X'$ are bimeromorphic 
and for quite some time it was hoped that  $H^2(X,\ZZ)$  would at least determine the bimeromorphic type of $X$.
As two bimeromorphic K3 surfaces are always isomorphic,  a result of this type would still qualify as a true 
 generalization of the Global Torelli theorem for K3 surfaces. This hope was shattered by Namikawa's
example in
 \cite{Nam}:
 
 \smallskip
 
 $\bullet$ {\it There exist compact hyperk\"ahler manifolds $X$ and $X'$
(projective and of dimension four) with isometric Hodge structures $H^2(X,\ZZ)\cong H^2(X',\ZZ)$ but
without $X$ and $X'$ being bimeromorphic (birational).}

\smallskip

Nevertheless, at least for the time being the second cohomology of a compact hyperk\"ahler manifold is still believed to encode
most of the geometric information on the manifold. Possibly other parts of the cohomology might have to be added,
but no convincing general version of a conjectural Global Torelli theorem using more than the second cohomology
has been put forward so far.  

At the moment it seems unclear what the existence of a Hodge isometry $H^2(X,\ZZ)\cong
H^2(X',\ZZ)$ between two compact hyperk\"ahler manifolds could mean concretely
for the relation between the geometry of $X$ and $X'$ (but  see Corollary \ref{cor:CorHilb} for special examples). 
However, rephrasing the classical Torelli theorem for K3 surfaces in terms of moduli spaces
 suggests a result that  was eventually proved by Verbitsky in \cite{Verb}. 
 
\subsection{Global Torelli via moduli spaces}\label{sect:GTmoduli}

The following rather vague discussion is meant to motivate the main result of \cite{Verb} to
be stated in the next section. The missing details and precise definitions will be given later.

\subsubsection{} We start by rephrasing the Global Torelli theorem  for K3 surfaces  using  the mo\-duli space
$\gm$  of marked K3 surfaces and  the period map $\kp:\gm\to \PP(\Lambda\otimes\CC)$.
A marked K3 surface $(S,\phi)$ consists
of a K3 surface $S$ and an isomorphism of lattices $\phi:H^2(S,\ZZ)\congpf \Lambda$,
where $\Lambda:=2(-E_8)+3U$ is the unique even unimodular  lattice of signature $(3,19)$.
Two marked K3 surfaces $(S,\phi)$ and $(S',\phi')$
are isomorphic if there exists an isomorphism (i.e.\ a biholomorphic map)
$g:S\congpf S'$ with $\phi\circ g^*=\phi'$. Then by definition $\gm=\{(S,\phi)\}/_{\cong}$.

\smallskip

The Global Torelli theorem for K3 surfaces is equivalent to the following
statement.

\smallskip

$\bullet$ {\it The moduli space $\gm$ has two connected components
interchanged by $(S,\phi)\mapsto (S,-\phi)$ and the period map 
$$\kp:\gm\to D_\Lambda:=\{x\in\PP(\Lambda\otimes\CC)~|~x^2=0,~(x.\bar x)>0\},~
(S,\phi)\mapsto [\phi(H^{2,0}(S))]$$ is generically injective on each of the two components.}

\smallskip

\begin{rema}\label{rem:gen}
Injectivity really only holds generically, i.e.\ for $(S,\phi)$ in the complement of a countable
union of hypersurfaces (cf.\ Remark \ref{rem:perioddense}). This is  related to the aforementioned stronger form of the Global Torelli theorem
being valid only for generic K3 surfaces.
\end{rema}

Let us now consider the natural action
$${\rm O}(\Lambda)\times\gm\to\gm,~(\varphi,(S,\phi))\mapsto (S,\varphi\circ\phi).$$
 For any $(S,\phi)\in \gm^o$ in a connected component $\gm^o$ of $\gm$ the
subgroup of ${\rm O}(\Lambda)$ that fixes $\gm^o$ is $\phi\circ{\rm Mon}(X)\circ\phi^{-1}$, where  the
monodromy group ${\rm Mon}(S)\subset{\rm O}(H^2(S,\ZZ))$ is by definition generated by all monodromies 
$\pi_1(B,t)\to {\rm O}(H^2(S,\ZZ))$  induced by arbitrary smooth proper families $\kx\to B$ with $\kx_t=S$.

The transformation $-\id\in{\rm O}(\Lambda)$ induces the involution $(S,\phi)\mapsto (S,-\phi)$  that interchanges
the two connected components and, as it turns out, there is essentially no other $\varphi\in{\rm O}(\Lambda)$ 
with this property. This becomes part of the following reformulation of the Global Torelli theorem for K3 surfaces:
\smallskip

$\bullet$ {\it Each connected component $\gm^o\subset\gm$ maps generically injectively into $D_\Lambda$ and
for any K3 surface $S$ one has ${\rm O}(H^2(S,\ZZ))/{\rm Mon}(S)=\{\pm 1\}$.}\label{pref:GTK3}

\medskip

In order to show that this version implies the one above, one also needs the rather easy fact that
 any two K3 surfaces $S$ and $S'$ are deformation equivalent, i.e.\ that there exist a smooth proper family
$\kx\to B$ over a connected base and points $t,t'\in B$ such that $S\cong\kx_t$ and $S'\cong \kx_{t'}$. In particular,
all K3 surfaces are realized by complex structures on the same differentiable manifold.

\subsubsection{} Let us try to generalize the above discussion  to higher dimensions.
Restricting to compact hyperk\"ahler manifolds $X$  of a fixed deformation class,  the isomorphism type, say
$\Lambda$, of the lattice 
realized by the Beauville--Bogomolov form on $H^2(X,\ZZ)$
is unique, cf.\  Section \ref{sect:Defi}. So the moduli space $\gm_\Lambda$ of $\Lambda$-marked compact
hyperk\"ahler manifolds of fixed deformation type (the latter condition is
not reflected by the notation) can be introduced, cf.\ Section \ref{sect:msmarked} for details.

\smallskip

For the purpose of  motivation let us consider the following two statements. Both are false (!) \!in general, but the
important point here is that they are equivalent and that the first half of the second one turns out to be true.

\smallskip

$\bullet$ (Global Torelli, standard form) {\it Any Hodge isometry $H^2(X,\ZZ)\cong H^2(X',\ZZ)$
between generic $X$ and $X'$ can be lifted up to sign to an isomorphism  $X\cong X'$.}

\smallskip

$\bullet$ (Global Torelli, moduli version)  i) {\it On each   connected component 
$\gm_\Lambda^o\subset\gm_\Lambda$ 
   the period map $\kp:\gm_\Lambda^o\to D_\Lambda$ is generically injective.}
 ii)  {\it For any hyperk\"ahler manifold $X$ parametrized by $\gm_\Lambda$ one has
${\rm O}(H^2(X,\ZZ))/{\rm Mon}(X)=\{\pm 1\}$.}

\smallskip

\begin{rema}\label{rem:passagebim} In both statements, generic is meant in the sense of Remark \ref{rem:gen}. 
The standard form would then indeed imply that arbitrary Hodge isometric
$X$ and $X'$ are bimeromorphic. For details on the passage from generic to arbitrary
hyperk\"ahler manifolds and thus to the bimeromorphic version of the Global Torelli theorem,
see Section \ref{sec:GT1}.
\end{rema}

Only rewriting the desired  Global Torelli theorem in its moduli version allows one to pin
down the reason for its failure in higher dimensions: As shall be explained below, condition {\bf i)} is always
fulfilled and this is the main result of \cite{Verb}. It is condition {\bf ii)} which need not hold. Indeed, 
a priori the image of the natural action $\Diff(X)\to{\rm O}(H^2(X,\ZZ))$  can have index larger than two and hence ${\rm Mon}(X)$ will.

\subsection{Main result}

The following is a weaker version of the main result of  \cite{Verb}.
Teichm\"uller spaces are here replaced by the more commonly
used moduli spaces of marked manifolds. Additional problems occur
in the Teichm\"uller setting which have been addressed in a more recent version of
\cite{Verb} (see Remark \ref{rema:Teichprob}).

\begin{theo}[Verbitsky]\label{theo:Verb}  
Let $\Lambda$ be a lattice of signature $(3,b-3)$ and let
$\gm_\Lambda^o$ be a connected component of the moduli space  $\gm_\Lambda$  of 
marked compact hyperk\"ahler manifolds $(X,\phi:H^2(X,\ZZ)\congpf \Lambda)$. 
Then the period map 
$$\kp:\gm_\Lambda^o\to D_\Lambda\subset \PP(\Lambda\otimes\CC),~(X,\phi)\mapsto[\phi(H^{2,0}(X))]$$
is generically injective.
\end{theo}

More precisely, all fibers over points in the complement of the countable union
of hyperplane sections $D_\Lambda\cap\bigcup_{0\ne\alpha\in\Lambda}\alpha^\perp$ 
consist of exactly one point. For the precise definition of the moduli space
of marked hyperk\"ahler manifolds $\gm_\Lambda$, the period map~$\kp$, and
 the period domain $D_\Lambda$ we refer to the text.

\smallskip

The starting point for the proof of the theorem are the following results:

\smallskip

--  {\it The period map $\kp$ is \'etale, i.e.\ locally an isomorphism of complex manifolds.
This is the Local Torelli theorem, see \cite{BeauvJDG} and Section \ref{sec:LT}.}

\smallskip

-- {\it The period domain $D_\Lambda$ is known to be simply connected. This is a standard fact, 
see Section \ref{subsec:perioddomain}.}

\smallskip

-- {\it The period map $\kp$ is surjective on each connected component. The surjectivity of the period map
has been proved in \cite{HuyInv}, see Section \ref{sec:Globsurj}.}

\smallskip

These three facts suggest that $\gm_\Lambda$ may be a covering space of the simply connected period domain
$D_\Lambda$, which would immediately show that each connected component maps isomorphically
onto $D_\Lambda$. There are however two issues that have to be addressed:

\smallskip

$\bullet$ {\it The moduli space $\gm_\Lambda$ is a complex manifold, but it is not Hausdorff.}

$\bullet$ {\it Is the period map $\kp:\gm_\Lambda\to D_\Lambda$ proper?}

\smallskip

Verbitsky deals with both questions.  First one passes from $\gm_\Lambda$ to a Hausdorff
space $\overline\gm_\Lambda$ by identifying all inseparable points in $\gm_\Lambda$.
The new space $\overline\gm_\Lambda$  still maps to $D_\Lambda$ via the period map.
This first step is technically involved, but it is the properness of  $\kp:\overline\gm_\Lambda\to D_\Lambda$
that is at the heart of  Theorem \ref{theo:Verb}. 

\medskip

In this text we give  a complete and rather detailed proof of Verbitsky's theorem, following the general approach
of \cite{Verb}. Some of the arguments have been simplified and sometimes (e.g.\ in Section \ref{sect:periodstwistor}) we
chose to apply more classical techniques. Our presentation of the material is very close to
Beauville's account  of the theory of K3 surfaces in \cite{BeauvAst}. Unfortunately, due to time and space restrictions we
will not be able to discuss the interesting consequences of Verbitsky's theorem in any detail. Some
of the beautiful applications will be touched upon in Section \ref{sec:GT}.

\medskip

 {\bf Acknowledgments:} 
 For comments and questions on an early version of the text I wish to thank
 A.\ Beauville, O.\ Debarre, P.\ Deligne, R.\ Friedman, K.\ Hulek, K.~O'Grady, and M.\ Rapoport.  
I am most grateful to E.\ Markman. The presentation of the material
 is inspired by our discussions on preliminary versions of \cite{Verb}. He made me aware
 of various technical subtleties in \cite{Verb} and  suggested  improvements to this
 note.\footnote{This text was prepared  while enjoying the  hospitality and financial support of the
Mathematical Institute Oxford. The author is a member of   the SFB/TR 45 `Periods,
Moduli Spaces and Arithmetic of Algebraic Varieties' of the DFG.}
\section{Recollections}\label{sect:Defi}

We briefly recall the main definitions and facts. For an introduction with more details and references, see
e.g.\  \cite{GHJ}.

\begin{defi}
A compact hyperk\"ahler (or irreducible holomorphic symplectic)
manifold is a simply connected compact  complex manifold  of K\"ahler type $X$
such that $H^0(X,\Omega_X^2)$ is spanned by an everywhere non-degenerate two-form $\sigma$. 
\end{defi}

As long as no hyperk\"ahler metric is fixed,  one should maybe, more accurately,  speak of compact
complex manifolds of hyperk\"ahler type but we will instead mention explicitly when a hyperk\"ahler metric
is chosen. Recall that by Yau's solution of the Calabi conjecture, any
K\"ahler class $\alpha\in H^{1,1}(X,\RR)$ can be uniquely represented by the K\"ahler form
of a Ricci-flat K\"ahler metric $g$. In fact, under the above conditions on $X$, the holonomy of such a metric
is ${\rm Sp}(n)$, where $2n=\dim_\CC(X)$. In particular, besides the complex structure $I$ defining
$X$, there exist complex structures $J$ and $K$ satisfying the usual relation $K=I\circ J=-J\circ I$ and such that
$g$ is  also K\"ahler with respect to them. As will be recalled in Section \ref{sect:twistorRicci},
there is in fact a whole sphere of complex structures compatible with $g$.

\begin{exem}
K3 surfaces are the two-dimensional hyperk\"ahler manifolds. Recall that a compact complex
surface $S$ is a K3 surface, if the canonical bundle $\Omega_S^2$ is trivial and $H^1(S,\ko_S)=0$. It is not difficult to prove that
K3 surfaces are in fact simply connected. That they are also of K\"ahler type, a result due to Siu, is much deeper,
see \cite{BeauvAst} for the proof and references.
\end{exem}
 
The second cohomology $H^2(X,\ZZ)$ of a generic compact K\"ahler manifold can be
endowed with a quadratic form, the Hodge--Riemann pairing, which  in dimension $>2$
depends on the choice of a K\"ahler class and, therefore, is usually not integral.
For a compact hyperk\"ahler manifold $X$ the situation is much better. There exists a primitive integral
quadratic form $q_X$ on $H^2(X,\ZZ)$, the \emph{Beauville--Bogomolov form}, the main properties of which can be
summarized as follows:

\smallskip

$\bullet$ {\it $q_X$ is non-degenerate of signature $(3,{b}_2(X)-3)$.}

\smallskip

$\bullet$ {\it There exists a positive constant $c$ such that  $q(\alpha)^n=c\int_X\alpha^{2n}$ 
for all classes \mbox{$\alpha\in H^2(X,\ZZ)$,} i.e.\ up to scaling $q_X$ is a root of the top intersection
product on $H^2(X,\ZZ)$.}

\smallskip

$\bullet$ {\it The  decomposition $H^2(X,\CC)=(H^{2,0}\oplus H^{0,2})(X)\oplus H^{1,1}(X)$ is orthogonal with respect to
(the $\CC$-linear extension of) $q_X$. Moreover, $q_X(\sigma)=0$ and $q_X(\sigma,\bar\sigma)>0$.}

\smallskip

The second property ensures that $q_X$  is invariant under deformations, i.e.\ if $\kx\to B$ is a smooth 
family of compact hyperk\"ahler manifolds over a connected base $B$, then $q_{\kx_t}=q_{\kx_s}$ for
all fibres $\kx_s,\kx_t\subset\kx$. Here, an isomorphism $H^2(\kx_s,\ZZ)\cong H^2(\kx_t,\ZZ)$ 
is obtained by parallel transport along a path connecting $s$ and $t$.
 In fact, at least for $b_2(X)\ne 6$ the primitive quadratic form $q_X$ 
only depends on the underlying differentiable manifold. 

The last property shows that the weight-two Hodge structure on $H^2(X,\ZZ)$ endowed with $q_X$ is uniquely determined by the line $\sigma\in H^{2,0}(X)\subset
H^2(X,\CC)$.

Note that the lattice $\Lambda$ defined by $(H^2(X,\ZZ),q_X)$ is in general not unimodular and there is no reason
why it should always be even (although in all known examples it is).
No classification of lattices that can be realized by the Beauville--Bogomolov form on some
compact hyperk\"ahler manifold is known. Also note that
no examples of compact hyperk\"ahler manifolds are known that would realize the same lattice
without being deformation equivalent and hence diffeomorphic.
For a K3 surface $S$ the Beauville--Bogomolov form coincides with the intersection form on $H^2(S,\ZZ)$
which is isomorphic to the unique even unimodular lattice of signature $(3,19)$.

\begin{exem}
i) The Hilbert scheme (or Douady space) $\Hilb^n(S)$ parametrizing subschemes $Z\subset S$ of length $n$ 
in a K3 surface $S$ is a compact hyperk\"ahler manifold of dimension $2n$ (see \cite{BeauvJDG}).  Moreover,
$(H^2(\Hilb^n(S),\ZZ),q)\cong H^2(S,\ZZ)\oplus \ZZ(-2(n-1))$ 
for $n>1$ and in particular $b_2(\Hilb^n(S))=23$. Note that $\Hilb^n(S)$ can be deformed to
compact hyperk\"ahler manifolds which are not isomorphic to the Hilbert scheme of any K3 surface.

ii) If $T$ is a two-dimensional complex torus $\CC^2/\Gamma$ and $\Sigma:\Hilb^n(T)\to T$  is the morphism
induced by the additive structure of $T$, then the \emph{generalized Kummer variety}
$K^{n-1}(T):=\Sigma^{-1}(0)$ is a compact hyperk\"ahler manifold of dimension $2n-2$.
In this case, $(H^2(K^{n-1}(T),\ZZ),q)\cong H^2(T,\ZZ)\oplus \ZZ(-2n)$  for $n>2$ and thus $b_2(K^{n-1}(T))=7$.
Again, the generic deformation
of $K^{n-1}(T)$ is a compact hyperk\"ahler manifold not isomorphic to any generalized Kummer variety itself.

iii) The only other known examples were constructed by O'Grady, see \cite{OG1,OG2}. They are of dimension six, resp.\ ten.
\end{exem}

\section{Period domain and twistor lines}\label{sect:periodstwistor}

\subsection{Period domain}\label{subsec:perioddomain}
Consider a non-degenerate  lattice  $\Lambda$ with a quadratic form $q$ (not necessarily unimodular or even) 
of signature $(3,b-3)$. Later $\Lambda$ will be $H^2(M,\ZZ)$ of a hyperk\"ahler manifold
$M$ endowed with the Beauville--Bogomolov form.

The  \emph{period domain} associated to $\Lambda$ is the set
$$D:=D_\Lambda:=\{x\in\PP(\Lambda\otimes\CC)~|~q(x)=0\text{~~and~~} q(x,\bar x)>0\}.$$ 
The quadratic form $q$ is extended $\CC$-linearly.
Thus, $D$  is an open subset of the smooth quadric hypersurface defined by $q(x)=0$ and,
in particular, $D$ has the structure of a complex manifold of dimension $b-2$ which is obviously Hausdorff.

The global structure of the period domain $D$ is also well-known.
Let ${\rm Gr}^{\rm po}(2,V)$ be the Grassmannian of oriented positive planes in 
a real vector space $V$ endowed with a quadratic form and consider $\RR^b$
with the diagonal quadratic form $(+1,+1,+1,-1,\ldots,-1)$. 
See e.g.\  \cite[Exp.\ VIII]{BeauvAst} for the proof of the following.

\begin{prop}\label{prop:topD}
There exist diffeomorphisms
$$D\cong {\rm Gr}^{\rm po}(2,\Lambda\otimes\RR)\cong {\rm Gr}^{\rm po}(2,\RR^b)\cong{\rm O}(3,b-3)/\left({\rm SO}(2)\times {\rm O}(1,b-3)\right).$$
In particular, $D$ is connected with $\pi_1(D)=\{1\}$.\qqed
\end{prop}

The first isomorphism is given by mapping $x\in D$ to the plane $P(x)\subset \Lambda\otimes\RR$ spanned
by the real and imaginary part of $x$. For the converse, choose an orthonormal positive oriented
basis $a,b\in P$ of  a given plane $P\in {\rm Gr}^{\rm po}(\Lambda\otimes\RR)$  and map $P$ to $a+ib\in D$.
For the last two diffeomorphisms choose an isometry $\Lambda\otimes\RR\cong \RR^b$.
In the following, we shall often tacitly pass from one description to the other and will not
distinguish between the point $x\in D$ and its associated plane
$P(x)\in{\rm Gr}^{\rm po}(2,\Lambda\otimes\RR)$.

\begin{rema}\label{rem:perioddense}
For any $0\ne\alpha\in \Lambda$ one can consider the intersection $D\cap \alpha^\perp$, where
$\alpha^\perp$ is the hyperplane $\{x\in\PP(\Lambda\otimes\CC)~|~q(x,\alpha)=0\}$. Thus $D\cap\alpha^\perp\subset D$ is
non-empty of complex codimension one.

Moreover,  for any open subset $U\subset D$ the set
$U\setminus \bigcup_{0\ne\alpha\in\Lambda}\alpha^\perp$ is dense in $U$. Indeed, pick a generic one-dimensional
disk $\Delta$ through a given point $x\in U$. Then each $\alpha^\perp$  intersects $\Delta$ in finitely many points
and since $\Lambda$ is countable, the intersection $\Delta\cap\bigcup\alpha^\perp$ is at most countable.
But the complement of any countable set inside $\Delta$ is dense in $\Delta$.
\end{rema}

\subsection{Twistor lines}

A subspace $W\subset \Lambda\otimes\RR$ of dimension three such that $q|_W$ is positive definite
is called  a \emph{positive three-space}.

\begin{defi} For any positive three-space $W$ one defines
the associated twistor line $T_W$ as
the intersection $$T_W:=D\cap \PP(W\otimes\CC).$$
\end{defi}

For $W$  a positive three-space,  $\PP(W\otimes\CC)$ is a plane
in $\PP(\Lambda\otimes\CC)$ and $T_W$ is a smooth quadric in $\PP(W\otimes\CC)\cong\PP^2$.
Thus, as a complex manifold $T_W$ is simply $\PP^1$. 

Two distinct points $x,y\in D$ are contained in one
twistor line if and only if their associated positive planes $P(x)$ and $P(y)$ span a positive
three-space $\langle P(x),P(y)\rangle\subset\Lambda\otimes\RR$.

\begin{defi}
A twistor line $T_W$ is called generic if $W^\perp\cap\Lambda=0$. 
\end{defi}

\begin{rema}\label{rem:generic}
One easily checks that $T_W$ is generic if and only if there exists a vector $w\in W$ with $w^\perp\cap \Lambda=0$,
which is also  equivalent to the existence of a point $x\in T_W$ such that $x^\perp\cap\Lambda=0$.
In fact, if $W$ is generic, then for all except a countable number of points $x\in T_W$ one has
$x^\perp\cap \Lambda=0$ (cf.\ Remark \ref{rem:perioddense}).
\end{rema}

\begin{defi}
Two points $x,y\in D$ are called equivalent (resp.\ strongly equivalent)  if there exists a
chain of  twistor lines (resp.\ generic twistor line) $T_{W_1},\ldots, T_{W_k}$ and points $x=x_1,\ldots, x_{k+1}=y$ 
with $x_i,x_{i+1}\in T_{W_i}$.
\end{defi}

The following is well-known, see  \cite{BeauvAst}.

\begin{prop}\label{prop:allequiv}
Any two points $x,y\in D$ are (strongly) equivalent.
\end{prop}

\begin{proof}
Since $D$ is connected, it suffices to show that
each equivalence class is open. Let us deal with the weak version first.

\vskip-0.3cm

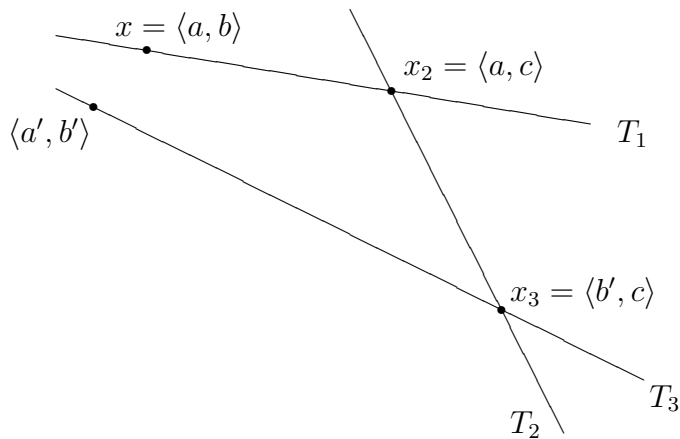
$$\hskip-2cm
\begin{picture}(-100,100)
\put(-150,80){\line(6,-1){200}}
\put(-40,90){\line(1,-2){80}}
\put(-150,60){\line(2,-1){220}}

\put(-128,80){\makebox{$x=\langle a,b\rangle$}}
\put(-168,40){\makebox{$\langle a',b'\rangle$}}
\put(-20,65){\makebox{$x_2=\langle a,c\rangle$}}
\put(20,-20){\makebox{$x_3=\langle b',c\rangle$}}

\put(60,40){\makebox{$T_1$}}
\put(20,-70){\makebox{$T_2$}}
\put(72,-60){\makebox{$T_3$}}

\put(-116,74.5){\circle*{3}}
\put(-24.5,59){\circle*{3}}
\put(16.7,-23.5){\circle*{3}}
\put(-136,53){\circle*{3}}
\end{picture}
$$

\vskip3cm

Consider $x\in D$ and choose a basis for the corresponding
positive plane   $P(x)=\langle a,b\rangle$. Here and in the sequel, the order of the basis vector
is meant to fix the orientation of the plane.
Pick $c$ such that $\langle a,b,c\rangle$ is a positive three-space.
Then for $(a',b')$ in an open neighbourhood of $(a,b)$ the spaces $\langle a,b',c\rangle$ and $\langle a',b',c\rangle$
are still positive three-spaces. 
Let $T_1,T_2$, and $T_3$ be the  twistor lines associated to
$\langle a,b,c\rangle$, $\langle a,b',c\rangle$, resp.\ $\langle a',b',c\rangle$. Then $P(x)=\langle a,b\rangle,
\langle a,c\rangle\in T_1$, $\langle a,c\rangle,\langle b',c\rangle\in T_2$, and $\langle b',c\rangle, \langle a',b'\rangle\in T_3$. Thus, $x$ and $\langle a',b'\rangle$ are connected via the chain
of the three  twistor lines $T_1,T_2$, and $T_3$.

For the strong equivalence, choose in the above argument $c$ such that $c^\perp\cap\Lambda=0$. Then the 
 twistor lines associated to the positive three-spaces $\langle a,b,c\rangle$, $\langle a,b',c\rangle$,
and $\langle a',b',c\rangle$
are all generic (see Remark \ref{rem:generic}).
\end{proof}

This easy observation is crucial for the global surjectivity of the period map (see Section \ref{sec:Globsurj}).
In order to prove
that the period map is a covering map, one also needs a local version of the surjectivity
(cf.\ Section \ref{sec:Locsurj}) which in turn relies
on a local version of Proposition \ref{prop:allequiv}. This shall be explained next.

\begin{conv}\label{conv}
In the following, we consider balls in $D$ and write $B\subset \bar B\subset D$ 
when $\bar B$ is a closed ball in a differentiable chart in $D$. In particular, $B$ will be the
open set of interior points in $\bar B$.
\end{conv}
 
\begin{defi}
Two points $x,y\in B\subset\bar B\subset D$ are called  equivalent (resp.\ strongly equivalent) as points in $B$
if there exist a chain of (generic) twistor lines $T_{W_1},\ldots,T_{W_k}$ and
points $x=x_1,\ldots,x_{k+1}=y\in B$  
such that $x_i,x_{i+1}$ are contained in the same connected component of $T_{W_i}\cap B$.
\end{defi}

Note that a priori two points $x,y\in B$ could be (strongly) equivalent as points in $D$ without being (strongly)
equivalent as points in $B$. In the local case, only strong equivalence will be used.

The local version of Proposition \ref{prop:allequiv} is the following.

\begin{prop}\label{prop:localstrong}
For a given ball $B\subset \bar B\subset D$ any two points $x,y\in B$ are strongly equivalent as points in $B$.
\end{prop}

\begin{proof}
One again shows that each equivalence class is open which together with the connectedness of the ball $B$ proves the result. The proof is a modification of the argument for Proposition \ref{prop:allequiv} and we shall
use the same notation. The main difference is that a  given positive plane
$\langle a,b\rangle$  is connected to any nearby point by a chain of four generic  twistor lines instead of just three.
As the proof given here deviates from the original more technical one in \cite{Verb}, we
shall spell out all the (mostly elementary) details.

Let $x\in B$ and let $a,b$ be an oriented basis of the associated plane $P(x)$. The open sets
$U_\varepsilon:=\{\langle a',b'\rangle~|~\|a-a'\|<\varepsilon,~\|b-b'\|<\varepsilon\}$ form
a basis of open neighbourhoods of $x\in D$. Here $\|~\|$ is an arbitrary fixed norm on the real vector space
$\Lambda\otimes\RR$. Strictly speaking, $U_\varepsilon$ is an open set of
planes, but when we write $\langle a',b'\rangle\in U_\varepsilon$ we implicitly mean
that the vectors $a',b'$ satisfy the two inequalities defining $U_\varepsilon$.

Fix $0<\varepsilon<1$ small enough such that $x\in U_{\varepsilon}\subset B$. 
Then there exist $\langle d,c\rangle\in U_{\varepsilon}$
such that $\langle a,b,c\rangle$, $\langle a,b,d\rangle$ are positive
three-spaces. Indeed, if an isometry  \mbox{$\Lambda\otimes\RR\cong \RR^b$} is chosen such that $(a,b)=(e_1,e_2)$,
then take $c=e_2+(\varepsilon/2)e_3$ and \mbox{$d=e_1+(\varepsilon/2)e_3$.} Here $e_1,\ldots,e_b$ is the standard basis
of $\RR^b$ endowed with the quadratic form ${\rm diag}(1,1,1,-1,\ldots,-1)$. Moreover,
after adding small generic vectors, we can assume that $c^\perp\cap\Lambda=0=d^\perp\cap\Lambda$.

To be a positive three-space is an open condition. Thus there exists
$0<\delta<\varepsilon$ such that for all $\langle a',b'\rangle\in U_\delta\subset U_\varepsilon\subset B$
the spaces $\langle a',b',c\rangle$ and $\langle a',b',d\rangle$ are still positive three-spaces.

\vskip1.5cm

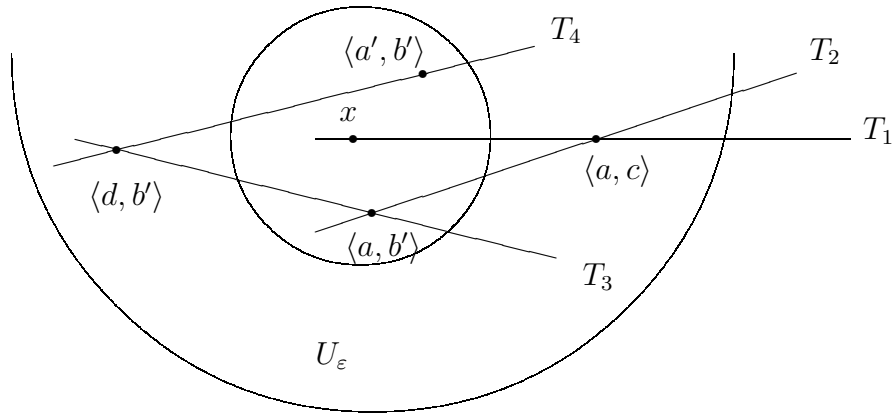
$$\hskip3cm
\begin{picture}(-100,100)
\put(-130,80){\line(6,0){200}}
\put(-130,45){\line(3,1){180}}
\put(-220,80){\line(4,-1){180}}
\put(-228,70){\line(4,1){180}}

\put(-121,87){\makebox{$x$}}
\put(-30,65){\makebox{$\langle a,c\rangle$}}
\put(-119,36){\makebox{$\langle a,b'\rangle$}}
\put(-215,55){\makebox{$\langle d,b'\rangle$}}
\put(-120,110){\makebox{$\langle a',b'\rangle$}}

\put(75,80){\makebox{$T_1$}}
\put(55,110){\makebox{$T_2$}}
\put(-30,25){\makebox{$T_3$}}
\put(-42,118){\makebox{$T_4$}}

\put(-116,80.){\circle*{3}}
\put(-25,80){\circle*{3}}
\put(-109,52){\circle*{3}}
\put(-204.5,76){\circle*{3}}
\put(-90,104.5){\circle*{3}}

\put(-190,30){\setlength{\unitlength}{0.9cm}
  \linethickness{0.075mm}
  \qbezier(4.9,2.0)(4.9,2.787)(4.3435,3.3435)
  \qbezier(4.3435,3.3435)(3.787,3.9)(3.0,3.9)
  \qbezier(3.0,3.9)(2.213,3.9)(1.6565,3.3435)
  \qbezier(1.6565,3.3435)(1.1,2.787)(1.1,2.0)
  \qbezier(1.1,2.0)(1.1,1.213)(1.6565,0.6565)
  \qbezier(1.6565,0.6565)(2.213,0.1)(3.0,0.1)
  \qbezier(3.0,0.1)(3.787,0.1)(4.3435,0.6565)
  \qbezier(4.3435,0.6565)(4.9,1.213)(4.9,2.0)} 
\put(-130,-5){\makebox{$U_{\varepsilon}$}}
\put(-322,-30){\setlength{\unitlength}{2.5cm}
  \linethickness{0.055mm}
  \qbezier(1.1,2.0)(1.1,1.213)(1.6565,0.6565)
  \qbezier(1.6565,0.6565)(2.213,0.1)(3.0,0.1)
  \qbezier(3.0,0.1)(3.787,0.1)(4.3435,0.6565)
  \qbezier(4.3435,0.6565)(4.9,1.213)(4.9,2.0)} 
\end{picture}
$$

\vskip1.8cm

For any given $\langle a',b'\rangle\in U_\delta$, let $T_1,\ldots,T_4$   be the generic (!) \!twistor lines
associated to the four positive three-spaces
$\langle a,b,c\rangle$, $\langle a,b',c\rangle$, $\langle a,b',d\rangle$, resp.\
$\langle a',b',d\rangle$. (Use that indeed
$\langle a,b\rangle,\langle a,b'\rangle,\langle a',b'\rangle\in U_\delta$.)

Let $x_1:=x=\langle a,b\rangle$, $x_2:=\langle a,c\rangle$, $x_3:=\langle a,b'\rangle$,
$x_4:=\langle d,b'\rangle$,
and $x_5:=\langle a',b'\rangle$. Then $x_i,x_{i+1}\in T_i\cap B$. This would show that $x$ and $\langle a',b'\rangle$
are strongly equivalent as points in~$B$ if  indeed $x_i$ and $x_{i+1}$
are contained in the same connected component of $T_i\cap B$ (in fact, as we will see,
of $T_i\cap U_{\varepsilon}$).
The verification of this is straightforward. E.g.\ $x_2=\langle a,c\rangle$ and $x_3=\langle a,b'\rangle$
can be connected via $\langle a,c+t(b'-c)\rangle$ with $t\in[0,1]$. This path is contained
in $T_2\cap B$ as $\|c+t(b'-c)-b\|=\|(1-t)(c-b)+t(b'-b)\|<(1-t)\varepsilon+t\delta\leq\varepsilon$.
\end{proof}

A much easier related observation is the following:

\begin{lemm}\label{lem:mucheasier}
Consider a ball $B\subset\bar B\subset D$ and let $x\in\bar B\setminus B$. Then there exists a generic twistor line $T_W\subset D$ such that $x\in\partial B\cap T_W$ is in the boundary of $B\cap T_W$. In other words, the boundary of $B$ can be connected to its interior
by means of generic  twistor lines.
\end{lemm}

\begin{proof}
Consider the tangent space $T_x \partial B$ of the boundary $\partial B$ in the point $x\in\partial B$. Then
$T_x\partial B$ is a subspace of real dimension $2b-5$ of the tangent space $T_xD$ which is of real
dimension $2b-4$. Describing  $D$ as the Grassmannian of oriented positive planes, yields 
a natural identification of  $T_xD$ with $\Hom(P(x),\Lambda\otimes\RR/P(x))$. 

Under this identification, the tangent space of a twistor line $T_W$ through $x$ corresponds
to the two-dimensional real subspace $\Hom(P(x),\RR\alpha)$, where we write $W=P(x)\oplus \RR\alpha$ 
for some positive $\alpha\in P(x)^\perp$ and
identify $\Lambda\otimes\RR/P(x)$ with $P(x)^\perp$. Conversely, the choice of $\alpha$ defines a twistor line through $x$ and if $\alpha$ is chosen generically then $\alpha^\perp\cap\Lambda=0$ and $\Hom(P(x),\RR\alpha)\not\subset T_x\partial B$, i.e.\ the corresponding $T_W$ is a generic twistor line through $x$ with $T_W\cap B\ne\emptyset$.
\end{proof}

\section{Period map}

\subsection{Local Torelli theorem}\label{sec:LT}

For any compact complex manifold $X$ there exists a versal deformation $\kx\to\Def(X)$.
As usual, $\Def(X)$ is understood as a germ of a complex space which can be chosen arbitrarily small.
Since $H^0(X,\kt_X)=0$ for $X$ a compact hyperk\"ahler manifold, the deformation
is in this case in fact universal. Moreover, since any small deformation of $X$ is again compact hyperk\"ahler,
one has $h^1(\kx_t,\kt_{\kx_t})= h^1(\kx_t,\Omega_{\kx_t})=h^{1,1}(\kx_t)\equiv {\rm const}$ and hence
$\kx\to \Def(X)$ is  universal for any of its fibers $\kx_t$. For a survey of Kuranishi's results on deformation theory
see e.g.\ \cite{Dou}.

Although $H^2(X,\kt_X)$ need not be trivial
for a compact hyperk\"ahler manifold, the base $\Def(X)$ is smooth of dimension $b_2(X)-2=h^1(X,\kt_X)$.
This is a result of Bogomolov \cite{Bog}, which can also be seen as a special case of the Tian--Todorov
unobstructedness result for K\"ahler manifolds with trivial canonical bundle.

Classical Hodge theory provides us with a holomorphic map $$\kp:\Def(X)\to \PP(H^2(X,\ZZ)),~t\mapsto
[H^{2,0}(\kx_t)]$$ for which one uses the canonical identification $H^2(X,\ZZ)\cong H^2(\kx_t,\ZZ)$
via parallel transport
which respects the Beauville--Bogomolov forms $q_X$ resp.\ $q_{\kx_t}$.
As $q_X(\sigma_{\kx_t})=q_{\kx_t}(\sigma_{\kx_t})=0$ and $q_X(\sigma_{\kx_t},\bar\sigma_{\kx_t})=
q_{\kx_t}(\sigma_{\kx_t},\bar\sigma_{\kx_t})>0$, the period map takes values in the \emph{period domain}
$$D_X:=D_{H^2(X,\ZZ)}\subset \PP(H^2(X,\CC)).$$ 
The following result was proved  in \cite{BeauvJDG}.

\begin{theo}[Local Torelli theorem]\label{thm:LT}
The period map $$\kp:\Def(X)\to D_X\subset \PP(H^2(X,\CC))$$
is  biholomorphic onto an open subset of the period domain $D_X$.\qqed
\end{theo}

\subsection{Moduli space of marked hyperk\"ahler manifolds}\label{sect:msmarked}

For any given non-degenerate lattice $\Lambda$ of signature $(3,b-3)$
one defines the \emph{moduli space of $\Lambda$-marked hyperk\"ahler manifolds}
as $$\gm_\Lambda:=\{(X,\phi)\}/_{\cong}.$$
Here, $X$ is a compact hyperk\"ahler manifold   and
$\phi:H^2(X,\ZZ)\cong\Lambda$ is an isometry between $H^2(X,\ZZ)$ endowed with the Beauville--Bogomolov pairing 
and the lattice~$\Lambda$.  Two $\Lambda$-marked hyperk\"ahler manifolds
$(X,\phi)$ and $(X',\phi')$ are isomorphic, $(X,\phi)\sim(X',\phi')$, if there exists a biholomorphic map
$g:X\congpf X'$ such that $\phi\circ g^*=\phi'$.

\begin{rema}For most lattices $\Lambda$ one expects $\gm_\Lambda=\emptyset$; at least very few lattices are
known that are realized.
On the other hand, in all known examples the lattice~$\Lambda$ determines the diffeomorphism type of $X$. The
latter suggests to actually fix the underlying real manifold $M$, to put $\Lambda=H^2(M,\ZZ)$, and to consider
$$\gm_M:=\{(X,\phi)~|~X\sim_{{\mbox{\tiny diff}}}M\}/_{\cong}\subset\gm_{\Lambda}$$ as the space of marked hyperk\"ahler manifolds $X$ diffeomorphic to $M$ (but without fixing the diffeomorphism). 
As we shall eventually restrict to a connected component of the moduli space and $\gm_M$ is a union of connected components of $\gm_{\Lambda}$,
one can work with either of the two moduli spaces,   $\gm_M$ or $\gm_\Lambda$.
\end{rema}

The following is a well-known generalization of a construction used for K3 surfaces
(see \cite{BeauvAst} or \cite[Prop.\ 7.7]{HuyHab}).

\begin{prop}
The moduli space of $\Lambda$-marked hyperk\"ahler manifolds $\gm_\Lambda$ has the structure of a complex
manifold of dimension $b-2$. For any $(X,\phi)\in\gm_\Lambda$, there is a natural
holomorphic map $\Def(X)\hookrightarrow \gm_\Lambda$ identifying $\Def(X)$ with an open neighbourhood
of $(X,\phi)$ in $\gm_\Lambda$.
\end{prop}

\begin{proof}
 The base of  the universal deformation  $\kx\to \Def(X)$ of a compact hyperk\"ahler manifold $X$
 parametrized by $\gm_\Lambda$  can be thought of as a small disk of dimension $b-2$.
A marking $\phi$ of $X$ naturally induces markings $\phi_t$ of all
the fibers $\kx_t$ and by the Local Torelli Theorem \ref{thm:LT}
the period map $\kp:\Def(X)\to D\subset\PP(\Lambda\otimes\CC)$ defined
by $\phi$ is a local isomorphism. Hence, for $t_0\ne t_1\in\Def(X)$ the fibers
$\kx_{t_0}$ and $\kx_{t_1}$ with the markings $\phi_{t_0}$ and $\phi_{t_1}$, respectively, are non-isomorphic as marked manifolds.

Thus, the base space $\Def(X)$ can be regarded as a subset of $\gm_\Lambda$
containing $(X,\phi)$. For $(X,\phi),(X',\phi')\in\gm_\Lambda$
consider the intersection $\Def(X)\cap\Def(X')$, which might of course be empty.
Since $\kx\to \Def(X)$ is a universal deformation for each of its fibers $\kx_t$ (and similarly for
$\kx'\to\Def(X')$), this is an open  subset
of $\Def(X)$ and $\Def(X')$ on which the two induced complex structures coincide. Therefore, the complex structures of the
deformation spaces $\Def(X)$ for all $X$ parametrized by $\gm_\Lambda$ glue to a complex structure on $\gm_\Lambda$.
Since $\Def(X)$ is smooth, also $\gm_\Lambda$ is smooth.
\end{proof}

\begin{rema}
Two words of warning are in order at this point. Firstly, $\gm_\Lambda$ is a complex manifold but
in general it is not  Hausdorff. Second, the universal families $\kx\to \Def(X)$ and $\kx'\to\Def(X')$
do not necessarily glue over the intersection $\Def(X)\cap\Def(X')$ in $\gm_\Lambda$. This is due to
the possible existence of automorphisms acting trivially on the second cohomology. See \cite{BeauvBirk}
for explicit examples.
\end{rema}

By the very construction of the complex structure on $\gm_\Lambda$, the local period
maps $\kp:\Def(X)\to D_X\subset\PP^2(H^2(X,\CC))$ glue to the \emph{global period map}
$$\kp:\gm_\Lambda\to \PP(\Lambda\otimes\CC).$$ The global period map takes values in the period
domain $D\subset\PP(\Lambda\otimes\CC)$ (see Section \ref{subsec:perioddomain}) and the Local Torelli
Theorem \ref{thm:LT} immediately gives

\begin{coro}
The
 period map $\kp:\gm_\Lambda\to D$ is locally biholomorphic.\qqed
\end{coro}

\subsection{The  moduli space  is made Hausdorff}

As alluded to before, the moduli space of marked hyperk\"ahler manifolds
$\gm_\Lambda$ need not be Hausdorff. This phenomenon is well-known for
K3 surfaces  and it cannot be avoided in higher dimensions either.

Recall that a topological space $A$ is \emph{Hausdorff} if for any two points $x\ne y\in A$ there exist
disjoint open sets $x\in U_x\subset A$ and $y\in U_y\subset A$. If $U_x\cap U_y\ne\emptyset$
for all open neighbourhoods $x\in U_x$ and $y\in U_y$, then $x$ and $y$ are called \emph{inseparable} and we write
$x\sim y$. Clearly, $x\sim x$ for all $x$ and $x\sim y$ if and only if $y\sim x$, i.e.\ $\sim$ is reflexive and symmetric.
But in general $x\sim y\sim z$ does not imply $x\sim z$, i.e.\ $\sim$ may fail to be transitive, in which
case it is not  an equivalence relation. 

Restricting to our situation at hand, we shall define an a priori stronger relation as follows.

\begin{defi}\label{def:approx}
For $x,y\in \gm_\Lambda$ with  $\kp(x)=\kp(y)\in D$ we say  $x\approx y$ if
there exist an open neighbourhood $U$ of $0:=\kp(x)=\kp(y)\in D$ and 
holomorphic sections $s_x,s_y$
of $\kp:\kp^{-1}(U)\to U$ such that:

{\rm  i)} $s_x=s_y$ on a dense open subset 
$U_0\subset U$ and

{\rm  ii)} $s_x(0)=x$ and $s_y(0)=y$.
\end{defi}

In order to show that $\sim$ and $\approx$ actually coincide, we need to recall the following result from \cite{HuyInv}.

\begin{prop}\label{prop:nonsepbirat}
Suppose $(X,\phi)$ and $(Y,\phi')$ correspond to inseparable distinct points
$x,y\in \gm_\Lambda$. Then $X$ and $Y$ are bimeromorphic and $\kp(x)=\kp(y)$ is contained
in $D\cap\alpha^\perp$ for some $0\ne\alpha\in \Lambda$.
\end{prop}

\begin{proof}
The first part is  \cite[Thm.\ 4.3]{HuyInv}. The bimeromorphic correspondence is
constructed roughly as follows. If $x\sim y$, then there exists a sequence $t_i\in \gm_\Lambda$ converging 
simultaneously to $x$ and $y$. For the universal families $\kx$ and $\ky$ of $X$, resp.\ $Y$, this
corresponds to isomorphisms $g_{i}:\kx_{t_i}\to\ky_{t_i}$  compatible with the markings
of $\kx_{t_i}$ and $\ky_{t_i}$.
Then the graphs $\Gamma_{g_i}$ are shown to converge to a cycle $\Gamma=Z+\sum Y_k\subset X\times Y$
of which the component $Z$  defines a bimeromorphic correspondence and the  components
$Y_k$ do not dominate neither of the two factors.

If $Z$ is not the graph of an isomorphism, then  the image in $X$
of  curves contracted by $Z\to Y$ describes curves in $X$. Thus $H^{2n-1,2n-1}(X,\ZZ)\ne0$ and
by duality also $H^{1,1}(X,\ZZ)\ne0$. Hence there exists a class $0\ne\alpha\in \Lambda$
with $\phi^{-1}(\alpha)\in H^{1,1}(X)$ and, therefore, $\kp(x)\in D\cap \alpha^\perp$.

Suppose $Z$ is the graph of an isomorphism. Consider the action of
 $[Z]_*+\sum[Y_k]_*$  on $\Lambda$
(via the given markings $\phi$ and $\phi'$). If the image of some $Y_k$ in $X$ and $Y$ is of codimension
$\geq 2$, then $[Y_k]_*$ acts trivially on $\Lambda$. If this is the case for all $[Y_k]$, then
$[Z]_*=[\Gamma]_*=[\Gamma_{g_i}]_*$ on $\Lambda$
and, since the $g_i$ are compatible with the markings, the latter is in fact the identity. But then
$Z$ is the graph of an isomorphism $X\cong Y$ that is compatible with the markings $\phi,\phi'$ and, therefore,
$x=y$. Contradiction.

If one of the $Y_k$ maps onto a divisor in $X$ or $Y$, then $H^{1,1}(X,\ZZ)\ne0$ or, equivalently,
$H^{1,1}(Y,\ZZ)\ne0$. So again in this case $\kp(x)=\kp(y)\in D\cap\alpha^\perp$ for some $0\ne\alpha\in \Lambda$.
\end{proof}

\begin{rema}\label{rema:Teichprob}
If one tries to apply the same argument to the Teichm\"uller space $\Teich(M)$, then 
one needs to show the following: If $Z$ defines an isomorphism and the $Y_i$ have images of 
codimension $\geq2$, then the isomorphism defined by $Z$ is in fact given by a diffeomorphism in 
the identity component $\Diff(M)_0$.
This issue was addressed in a more recent version of \cite{Verb}.

Note that for K3 surfaces  $Z$ always defines an isomorphism and that, 
a priori,  in higher dimensions   $Z$ could define an isomorphism without
any of  the components $Y_i$ mapping onto a divisor. 
\end{rema} 
 
\begin{prop}\label{lem:simapprox}
{\rm  i)} $\approx$ is an open equivalence relation.

{\rm  ii)}  $x\approx y$ if and only if $x\sim y$.

{\rm  iii)} $\sim$ is an open equivalence relation. 
\end{prop}

\begin{proof} i)  
Again, $\approx$ is reflexive and symmetric by definition. Let us show that
it is also transitive. Assume $x\approx y\approx z$ and choose  for $x\approx y$ a neighbourhood
$0:=\kp(x)=\kp(y)\in U\subset D$ and sections $s_x,s_y$
as in Definition \ref{def:approx}. Similarly for $y\approx z$, choose a neighbourhood $0=\kp(y)=\kp(z)\in U'\subset D$ and sections 
$t_y,t_z$. Replacing $U$ and $U'$ by their intersection, we may in fact assume $U=U'$.

Then $s_y(U)$ and $t_y(U)$ are both open neighbourhoods of $y$ and 
hence $s_y(U)\cap t_y(U)$ is. Since $\kp$ is a local homeomorphism, this also shows that
$s_y$ and $t_y$ coincide on the open neighbourhood $\tilde U:=\kp(s_y(U)\cap t_y(U))\subset U$ of $0$.
Since $s_x$ and $s_y$ coincide on a dense open subset of $U$, they also coincide on a dense
open subset of $\tilde U$. Similarly for $t_y$ and $t_z$.
Together with $s_y|_{\tilde U}=t_y|_{\tilde U}$ this    shows $x\approx z$. 
 
Recall that an equivalence relation $\approx$ on a topological space $A$
is open if the projection $A\to A/_\approx$ is open.
Equivalently, an equivalence relation
$\approx$ is open if for all $x\approx y$ and any open neighbourhood $x\in V_x\subset A$, there exists
an open neighbourhood $y\in V_y\subset A$ such that for any $y'\in V_y$ one finds an $x'\in V_x$ with $x'\approx y'$.
In our case, let $V_y:=s_y(\kp(V_x)\cap U)$. Indeed, the dense open subset $U_0$ 
on which $s_x=s_y$ (see Definition \ref{def:approx},~ii)) intersects the image of $\kp(V_y)$ in
a dense open subset and hence $s_x(\kp(y'))\approx s_y(\kp(y'))$ for all $y'\in V_y$.

 \smallskip

ii) As $x\approx y$ clearly implies $x\sim y$, we only need to show the converse.
So let $x\sim y$. Then $0:=\kp(x)=\kp(y)$. Pick an open neighbourhood $0\in U\subset D$
of $0$ and holomorphic sections $s_x,s_y:U\to \gm_\Lambda$ with $s_x(0)=x$ and $s_y(0)=y$. 
Since $x\sim y$, the  intersection $V:=s_x(U)\cap s_y(U)$ cannot be empty.

In order to show that
$x\approx y$, it suffices to show that the open subset
$U_0:=\kp(V)$ is dense in $U$. If for $x'\in s_x(U)$ and $y'\in s_y(U)$ one has
$t:=\kp(x')=\kp(y')\in\partial U_0$
(the boundary of $U_0$ in $U$), then $x'\sim y'$ and $x'\ne y'$.

By Proposition \ref{prop:nonsepbirat} this implies that $t$ is contained in $D\cap \alpha^\perp$ for some class
$0\ne \alpha\in \Lambda$. Hence $\partial U_0\subset \bigcup_{0\ne\alpha\in \Lambda} \alpha^\perp$.
This is enough to conclude that $U_0$ is dense in $U$. Indeed, suppose $ U\setminus \overline U_0\ne\emptyset$. Then
connect a generic point in $U_0$ via a one-dimensional disk $\Delta\subset U$
with a generic point in the open subset $U\setminus \overline U_0$. Then
$\Delta\cap \partial U_0\subset\Delta\cap\bigcup_{0\ne\alpha\in \Lambda} \alpha^\perp$
is countable and can therefore not separate the two disjoint open sets $\Delta\cap U_0$ 
and $\Delta\cap (U\setminus \overline U_0)$. Contradiction. (Compare the arguments with
Remark \ref{rem:perioddense}.)

\smallskip

Obviously, iii) follows from i) and ii).
\end{proof}

\begin{coro}\label{cor:Hausms}
The period map  $\kp:\gm_\Lambda\to D\subset \PP(\Lambda\otimes\CC)$
factorizes over the  `Hausdorff  reduction' $\overline\gm_\Lambda$ of $\gm_\Lambda$.
More precisely, there exist a  complex Hausdorff manifold $\overline\gm_\Lambda$
and locally biholomorphic maps factorizing the period map:
$$\kp:\gm_\Lambda\twoheadrightarrow\overline\gm_\Lambda\to D,$$
such that $x=(X,\phi),y=(Y,\phi')\in\gm_\Lambda$ map to the same point in $\overline\gm_\Lambda$ if and only
if they are inseparable points of $\gm_\Lambda$, i.e.\ $x\sim y$.
\end{coro}

\begin{proof}  Consider the closure $R:=\bar\Delta$ of the diagonal
$\Delta\subset \gm_\Lambda\times \gm_\Lambda$.
Clearly, $R$ is the set of all tuples $(x,y)$ with $x\sim y$ and thus by Proposition
\ref{lem:simapprox}, iii)  the graph of an equivalence relation.

It is known that for an open equivalence relation $\sim$ on a topological space $A$ the quotient $A/_\sim$ is Hausdorff
if and only if its graph $R\subset A\times A$ is  closed  (see \cite[\S8 ${\rm N^o}$3. Prop.\ 8]{BourbTop}). 
Since $\sim$ is an open equivalence relation due to Proposition \ref{lem:simapprox}, iii) and  $R=\bar\Delta$,
this shows that indeed  $\gm_\Lambda/_\sim$ is Hausdorff.

The period map $\kp:\gm_\Lambda\to D$ is a local homeomorphism and factorizes via $\gm_\Lambda
\to \gm_\Lambda/_\sim\to D$.
Hence also $\gm_\Lambda\to \gm_\Lambda/_\sim$ is a local homeomorphism which allows one to endow
$\gm_\Lambda/_\sim$ with
the structure of a complex manifold. 

So, $\overline\gm_\Lambda:=\gm_\Lambda/_\sim$ (together with the natural maps) has the required properties.
\end{proof}

\begin{rema} The same arguments apply to any connected component $\gm_\Lambda^o$
of $\gm_\Lambda$. Thus by identifying
inseparable points, one  again obtains a Hausdorff space $\overline\gm_\Lambda^o$. Since
points in distinct connected components of $\gm_\Lambda$ can always be separated,
$\overline\gm_\Lambda^o$ is in fact a connected component of $\overline\gm_\Lambda$.
\end{rema}


\subsection{Twistor deformation and lifts of twistor lines}\label{sect:twistorRicci}

We briefly recall the construction of the twistor space. For more details see e.g.\ \cite{GHJ,Hit}.
 
Any K\"ahler class $\alpha\in H^{1,1}(X,\RR)$ on a hyperk\"ahler manifold
$X=(M,I)$ is uniquely represented by a K\"ahler form $\omega=\omega_I=g(I(~),~)$ of a hyperk\"ahler metric $g$.
The hyperk\"ahler metric $g$ comes with a sphere of complex structures $\{\lambda=
aI+bJ+cK~|~a^2+b^2+c^2=1\}$, where $K=I\circ J=-J\circ I$. Each $(M,\lambda)$ is again a complex manifold
of hyperk\"ahler type with a distinguished K\"ahler form $\omega_\lambda:=g(\lambda(~),~)$
and a holomorphic two-form $\sigma_\lambda\in H^0((M,\lambda),\Omega^2_{(M,\lambda)})$.
E.g.\ for $\lambda=J$ the latter can be explicitly given as $\sigma_J=\omega_K+i\omega_I$.
In general, the forms $\omega_\lambda$, ${\rm Re}(\sigma_\lambda)$, and ${\rm Im}(\sigma_\lambda)$ 
are contained in the three-dimensional space spanned by $\omega_I$, ${\rm Re}(\sigma_I)$,
and ${\rm Im}(\sigma_I)$.

The \emph{twistor space} associated to $\alpha$
is the complex manifold $\kx$ described by the complex structure 
$\II\in{\rm End}(T_mM\oplus T_\lambda\PP^1)$, $(v,w)\mapsto
(\lambda(v),I_{\PP^1}(w))$ on the  differentiable manifold $M\times\PP^1$.
Here, $I_{\PP^1}$ is the standard complex structure on $\PP^1$.
The integrability of $\II$ follows from  the Newlander--Nirenberg theorem, see \cite{HKR}.
In particular, the projection defines a holomorphic
map $$\kx\to\PP^1$$ whose fiber over $\lambda=I$ is just $X=(M,I)$.
If one wants to stress the dependence on the K\"ahler class $\alpha$, one also
writes $\kx(\alpha)\to T(\alpha)\cong\PP^1$.

By construction, the twistor space is a family of complex structures
on a fixed mani\-fold $M$. Thus, if we take $\Lambda=H^2(M,\ZZ)$ endowed
with the Beauville--Bogomolov pairing, then the period map yields a holomorphic
 map $\kp:\PP^1\cong T(\alpha)\to D\subset \PP(\Lambda\otimes\CC)$.
 
In fact, the period map identifies $\PP^1\cong T(\alpha)$ with the twistor line
$T_{W_\alpha}\subset D$ associated to the positive three-space
$W_\alpha:=\langle[\omega_I],[{\rm Re}(\sigma_I)],[{\rm Im}(\sigma_I)]\rangle=\RR\alpha\oplus (H^{2,0}(X)\oplus H^{0,2}(X))_\RR$, i.e.\ 
$$\kp:\PP^1\cong T(\alpha)\congpf T_{W_\alpha}\subset D.$$

 \begin{rema}
Twistor spaces are central for the theory of K3 surfaces and higher-dimensional hyperk\"ahler manifolds.
In contrast to usual deformation theory, which only provides deformations of a hyperk\"ahler manifold
$X$ over some small disk, twistor spaces are global deformations. 
\end{rema}

\section{Global and local surjectivity of the period map}

\subsection{The K\"ahler cone of a generic hyperk\"ahler manifold}

Global and local surjectivity of the period map both rely on the following
result proved in \cite{HuyInv}.

\begin{theo}[K\"ahler cone]\label{thm:KK}
Let $X$ be a compact hyperk\"ahler manifold with ${\rm Pic}(X)=0$. Then
the K\"ahler cone ${\mathcal K}_X$ of $X$ is maximal, i.e.\ ${\mathcal K}_X={\mathcal C}_X$.
\end{theo}

Here, ${\mathcal C}_X$ is the \emph{positive cone}, i.e.\  the connected component of the
open cone \mbox{$\{\alpha\in H^{1,1}(X,\RR)~|~q(\alpha)>0\}$} that contains a K\"ahler class.

\begin{rema}\label{rem:KKarb}
For arbitrary compact hyperk\"ahler manifolds the K\"ahler cone can be described 
as the open set of classes in  ${\mathcal C}_X$ that are positive on all rational curves
(see e.g.\ \cite[Prop.\ 28.5]{GHJ}), but this stronger version will not be needed.

Theorem \ref{thm:KK} relies on the projectivity criterion for compact hyperk\"ahler manifolds
that shows that $X$ is projective if and only if ${\mathcal C}_X\cap H^2(X,\ZZ)\ne0$. The original
 proof in \cite{HuyInv} was incorrect. The corrected proof given in the Erratum to \cite{HuyInv}  uses
 the Demailly--Paun description
 \cite{DemPau} of the K\"ahler cone of an arbitrary compact K\"ahler manifold (see also
 \cite{DebBourb}).
\end{rema}

\begin{coro}\label{cor:perpperiod}
If $(X,\phi)\in\gm_\Lambda$, then $\Pic(X)=0$ if and only if
$\kp(X,\phi)\not\in\bigcup_{0\ne\alpha\in\Lambda}\alpha^\perp$. In this case, 
the K\"ahler cone of $X$ is maximal, i.e.\ $\kk_X=\kc_X$.
\end{coro}

\begin{proof}
The first part follows from the observation that $\phi^{-1}(\alpha)\in H^2(X,\ZZ)$ is of type $(1,1)$
if and only if it is orthogonal to the holomorphic two-form $\sigma_X$. This in turn is equivalent
to $\kp(X,\phi)\in\alpha^\perp$.
\end{proof}

\begin{prop}\label{prop:twistorlift}
Consider a marked hyperk\"ahler manifold
$(X,\phi)\in\overline\gm_\Lambda$ and assume that its period $\kp(X,\phi)$ is contained in
a  generic twistor line $T_W\subset D$. Then there exists
a unique lift of $T_W$ to a curve  in $\overline\gm_\Lambda$ through $(X,\phi)$, i.e.\  there exists a commutative diagram
$$\xymatrix{\overline\gm_\Lambda\ar[r]^-\kp& D\\
&\!\raisebox{-0.4cm}{$T_W$}\ar@{^{(}->}[u]_{i}\ar[lu]^{\tilde i}}$$
with $(X,\phi)$ in the image of $\tilde i$.
\end{prop}

\begin{proof}
Since $\kp:\overline\gm_\Lambda\to D$ is locally biholomorphic, 
the inclusion $i:\Delta\subset T_W\hookrightarrow D$
of a small open one-dimensional disk containing 
$0=\kp(X,\phi)\in\Delta$ 
can be lifted   to $\tilde i:\Delta\hookrightarrow \overline\gm_\Lambda$, $t\mapsto(X_t,\phi_t)$
with $\tilde i(0)=(X,\phi)$. By Corollary \ref{cor:Hausms} the space $\overline\gm_\Lambda$ is Hausdorff 
and hence    $\tilde i:\Delta\hookrightarrow \overline\gm_\Lambda$
is unique. (The uniqueness is a general fact from topology which works for any local homeomorphism between
Hausdorff spaces, see e.g.\ \cite[Lem.\ 1]{Brow}.)

As $T_W$ is a generic twistor line, the set $ T_W\cap\bigcup_{0\ne\alpha\in\Lambda}\alpha^\perp$ is countable
and thus for generic $t\in\Delta$ one has  $\Pic(X_t)=0$ (see Remark \ref{rem:generic} and Corollary \ref{cor:perpperiod}).
Let us fix such a generic $t$.

By construction,  $\phi_t(\sigma_t)\in W\otimes\CC$ and, therefore, there exists a class $\alpha_t\in H^2(X_t,\RR)$
such that $\phi_t(\alpha_t)\in W$ is orthogonal to
$\phi_t\langle{\rm Re}(\sigma_t),{\rm Im}(\sigma_t)\rangle\subset W$. Hence, $\alpha_t$ is of type $(1,1)$
on $X_t$ and $\pm\alpha_t\in\kc_{X_t}$, as  $W$ is a positive three-space. 
Due to Corollary \ref{cor:perpperiod} and using  $\Pic(X_t)=0$ for our fixed generic $t$,
this implies $\pm\alpha_t\in\kk_{X_t}$.
 
Now consider the twistor space  $\kx(\alpha_t)\to T(\alpha_t)$ for $X_t$ endowed with the K\"ahler class $\pm\alpha_t$.
Since $\phi_t\langle \alpha_t,{\rm Re}(\sigma_t),{\rm Im}(\sigma_t)\rangle=W$, the period map
yields an identification $\kp:T(\alpha_t)\congpf T_W$. 

Both, $T(\alpha_t)$ and $\tilde i(\Delta)$, contain the point $t$ and map locally isomorphically
to $T_W$. Again by the uniqueness of lifts for a local homeomorphism between Hausdorff spaces,
this proves $0\in T(\alpha_t)$ which yields the assertion.
\end{proof}

\subsection{Global surjectivity}\label{sec:Globsurj}

The surjectivity of the period map proved in \cite{HuyInv} is a direct consequence of the description of the K\"ahler cone of a generic hyperk\"ahler manifold.

\begin{theo}[Surjectivity of the period map]\label{thm:surj}
Let $\gm_\Lambda^o$ be a  connected component of the moduli space $\gm_\Lambda$
of marked hyperk\"ahler manifolds. Then the restriction of the period map
$$\kp:\gm_\Lambda^o\twoheadrightarrow D\subset\PP(\Lambda\otimes\CC)$$
is surjective.
\end{theo}

\begin{proof}
Since by  Proposition \ref{prop:allequiv} any two points
$x,y\in D$ are strongly equivalent, it is enough to show that $x\in\kp(\gm_\Lambda^o)$ if and only
if $y\in \kp(\gm_\Lambda^o)$  for any two points $x,y\in T_W\subset D$
contained in a generic twistor line $T_W$. This is an immediate consequence of Proposition
\ref{prop:twistorlift} which shows that the generic twistor line
$T_W$ can be lifted through any given preimage $(X,\phi)$ of $x$.
Indeed, then $y$ will also be contained in the image of the lift of $T_W$.
\end{proof}

\subsection{Covering spaces}

This section contains a criterion that decides when a local homeomorphism is a
covering space. Recall that a continuous map $\pi:A\to D$ between Hausdorff spaces
is a \emph{covering space}
if every point in $D$ admits an open neighbourhood
$U\subset D$ such that $\pi^{-1}(U)$ is the disjoint union $\coprod U_i$ of
open subsets $U_i\subset A$ such that the projections yields    homeomorphisms $\pi:U_i\congpf U$.

The reader may want to compare Proposition \ref{prop:Eyal} below with a
classical result of Browder \cite[Thm.\ 5]{Brow}
which says that a local homeomorphism $\pi:A\to D$ between topological  Hausdorff 
manifolds is a covering space if every point in $D$ has a neighbourhood~$U$ such that
$\pi$ is a closed map on each connected component of $\pi^{-1}(U)$.

\begin{prop}\label{prop:Eyal}
A local homeomorphism $\pi:A\to D$ between topological Hausdorff manifolds
is a covering space if for any
ball $B\subset\bar B\subset D$ (see  \ref{conv}) and any connected component
$C$  of the closed subset $\pi^{-1}(\bar B)$ one has $\pi(C)=\bar B$.
\end{prop}

\begin{proof}
We shall follow the alternative arguments of Markman given in the appendix to \cite{Verb}.
The techniques are again elementary, but need to be applied with care.

The proof can be immediately reduced to the case that $D=\RR^n$ and $A$ is connected.
Then $\pi:A\to D=\RR^n$ is a covering space, i.e.\ $\pi:A\congpf \RR^n$, if and only
if $\pi$ admits a section $\gamma:\RR^n\to A$. 

Pick a point $x\in A$ with $\pi(x)=0\in \RR^n$ (the origin) and
consider balls $B_\varepsilon\subset\bar B_\varepsilon\subset \RR^n$ of radius $\varepsilon$ centered in
$0\in\RR^n$. 
By the lifting property of local homeomorphisms (see e.g.\ \cite[Lem.\ 1]{Brow}) any
section $\gamma:\bar B_\varepsilon\to A$ is uniquely determined by $\gamma(0)$.

For small $0\leq\varepsilon$ there exists a section of $\pi$
over the closed ball $\gamma:\bar B_\varepsilon\to A$ 
 with $\gamma(0)=x$ (use that $\pi$ is a local homeomorphism in $x$).
 Let $I\subset \RR$ be the set of all $0\leq\varepsilon$ for which such a section exists.
 Then $I$ is a connected interval in $\RR_{\geq 0}$ containing $0$. It suffices to show that
 $I$ is open and closed, which would imply $I=[0,\infty)$ and thus prove the existence
 of a section $\gamma:\RR^n\to A$ of $\pi$.
 
 \smallskip
 
 {\it Claim: $I$ is open.}
 Consider $\varepsilon\in I$ and the corresponding section $\gamma:\bar B_\varepsilon\to A$ with
 $\gamma(0)=x$. Then choose for each point $t\in \bar B_\varepsilon\setminus B_\varepsilon$
 a small ball $B_{\varepsilon_t}(t)$ of radius $\varepsilon_t$ centered in $t$ over which $\pi$ admits a section $\gamma_t:B_{\varepsilon_t}(t)\to
 A$ with $\gamma_t(t)=\gamma(t)$. Note that then $\gamma=\gamma_t$ on the intersection
 $\bar B_\varepsilon\cap B_{\varepsilon_t}(t)$.
 
 Since $\bar B_\varepsilon\setminus B_\varepsilon$ is compact,
 there exist finitely many points $t_1,\ldots,t_k\in \bar B_\varepsilon\setminus B_\varepsilon$
 such that $\bar B_\varepsilon\setminus B_\varepsilon\subset\bigcup B_{\varepsilon_{t_i}}(t_i)$.
 Moreover, there also exists $\varepsilon<\delta$ such that $\bar B_\delta\subset B_\varepsilon\cup
 \bigcup B_{\varepsilon_{t_i}}(t_i)$. Then, $\gamma$ and the $\gamma_{t_i}$ glue to
 a section $\gamma:\bar B_\delta\to A$. Indeed, $\gamma$ and $\gamma_{t_i}$ coincide on 
 $B_\varepsilon\cap B_{\varepsilon_{t_i}}(t_i)$. In order to show that
  $\gamma_{t_i}$ and $\gamma_{t_j}$ glue over the intersection   $B_{\varepsilon_{t_i}}\cap B_{\varepsilon_{t_j}}$
  (if not empty), one uses that  this (connected) intersection 
also meets  $B_\varepsilon$ on which $\gamma_{t_i}$ and $\gamma_{t_j}$ both coincide with $\gamma$ and hence with each other. (Draw a picture!)    Hence $\delta\in I$ and thus $[0,\delta)\subset I$ is an open
 subset of $I$ containing $\varepsilon$.
 
 \smallskip
 
 {\it Claim: $I$ is closed.} In this step one uses the assumption $\pi(C)=\bar B$. Consider  $\delta \in \RR_{\geq0}$
 in the closure of $I$. Then for all $\varepsilon<\delta$
 there is a section $\gamma:\bar B_\varepsilon\to A$ with $\gamma(0)=x$. Therefore,
 there  exists a section over the open ball  $\gamma:B_\delta\to A$. 

Let $C_0$ be the closure  $\overline{\gamma(B_\delta)}\subset A$ and let us show that 
then $\pi:C_0\congpf\pi(C_0)$
and that $C_0$ coincides with the connected component $C$ of $\pi^{-1}(\bar B_\delta)$
containing $x$.

To do this, choose balls $B_{\varepsilon_t}(t)$ as in the previous step for
all points $t\in \bar B_\delta\setminus B_\delta$ that are also contained in $\pi(C_0)$.
We cannot apply a compactness argument, because a priori not every $t$ in the boundary of $B_\delta$
might be in the image of $C_0$. Nevertheless, the sections $\gamma_t$ and $\gamma$ glue to a
section $B_{\delta}\cup\bigcup_{t\in\pi(C_0)} B_{\varepsilon_t}(t)\to A$ and we denote the
image of this section
by $V\subset A$.

 Thus, $V$ is an open subset of $A$ homeomorphic to its image under $\pi$ (which 
is $B_{\delta}\cup\bigcup_{t\in\pi(C_0)} B_{\varepsilon_t}(t)$). But then $C_0=V\cap\pi^{-1}(\bar B_\delta)$
which in particular shows that $C_0$ is open in $\pi^{-1}(\bar B_\delta)$. By definition, $C_0$ is also closed and 
certainly contained in $C$. Hence $C_0$ coincides with the connected component $C$ and as $C_0\subset V\cong\pi(V)$, also $C_0\cong\pi(C_0)$.

Since by assumption $\pi(C)=\bar B_\delta$ and, as just proved, $C=C_0$, one finds that a section over $\bar B_\delta$ exists.
This  yields $\delta\in I$. Hence, $I$ is closed. 
\end{proof}

\subsection{Local surjectivity and proof of Verbitsky's theorem}\label{sec:Locsurj}

In this section we conclude the proof of Verbitsky's Theorem \ref{theo:Verb}.
The first step is a verification of the assumption of Proposition \ref{prop:Eyal},
which can be seen as a local version of the surjectivity of the period map (see Theorem
\ref{thm:surj}).

\begin{prop}
Consider the period map $\kp:\overline\gm_\Lambda^o\to D$ from a connected
component $\overline\gm_\Lambda^o$ of $\overline\gm_\Lambda$. If $B\subset\bar B\subset D$ is a ball (see \ref{conv}),
then for any connected component $C$ of $\kp^{-1}(\bar B)$ one has $\kp(C)=\bar B$.
\end{prop}

\begin{proof} We first adapt the arguments of  Theorem \ref{thm:surj} to show $B\subset \kp(C)$.

Clearly, $\kp(C)$ contains at least one point of $B$, because $\kp$ is a local homeomorphism.
Due to Proposition \ref{prop:localstrong}, any two points $x,y\in B$ are strongly equivalent as points in~$B$.
Thus it suffices to show that $x\in\kp(C)$ if  and only
if $y\in\kp(C)$ for any two points $x,y\in B$ contained in the same connected component of the intersection
$T_W\cap B$ with $T_W$ a generic twistor line. If $x=\kp(X,\phi)$ with $(X,\phi)\in C$, choose a local 
lift of the inclusion $x\in \Delta\subset T_W$ and then argue literally as in the proof of Theorem  \ref{thm:surj}.
The assumption that $x,y$ are contained in the same connected component of $T_W\cap B$ ensures
that the twistor deformation $T(\alpha_t)$ constructed there connects $(X,\phi)$ to a point over $y$ that is indeed again
contained in $C$.

It remains to prove that also the boundary $\bar B\setminus B$ is contained in $\kp(C)$. For this apply
Lemma \ref{lem:mucheasier} to any point $x\in \bar B\setminus B$ and lift the generic twistor line connecting~$x$ 
with a  point in $B$ to a twistor deformation as before.
\end{proof}

Then Proposition \ref{prop:Eyal} immediately yields:

\begin{coro} If $\overline\gm_\Lambda^o$ is a connected component of $\overline\gm_\Lambda$, then 
$\kp:\overline\gm_\Lambda^o\to D$ is a covering space.\qqed\end{coro}

Since $D$ is simply connected (see Proposition \ref{prop:topD}), this can equivalently be expressed as

\begin{coro}\label{cor:periodiso}
 If $\overline\gm_\Lambda^o$ is a connected component of $\overline\gm_\Lambda$, then 
$\kp:\overline\gm_\Lambda^o\to D$ is an isomorphism.\qqed
\end{coro}

The proof of Theorem \ref{theo:Verb} can now be completed as follows:

\smallskip

Consider a connected component $\gm^o_\Lambda$ of $\gm_\Lambda$. Then $\gm_\Lambda^o$ gives rise
to a connected component $\overline\gm_\Lambda^o$ of $\overline\gm_\Lambda$. By Corollary
\ref{cor:periodiso} the period map $\kp:\overline\gm_\Lambda^o\congpf D$ is an isomorphism and in particular
all its fibers  consist of exactly one point.

Thus it suffices to show that the generic fiber of the natural quotient
$$\pi:\gm_\Lambda^o\to\overline\gm_\Lambda^o$$ consists of just one point (see Remark \ref{rem:gen} for
the meaning of generic).
The fibers of $\pi$ are the equivalence classes of the equi\-valence relation $\sim$
or, equivalently, $\approx$ (see Corollary \ref{cor:Hausms}). By Proposition \ref{prop:nonsepbirat},
points with periods in the complement of $D\cap\bigcup_{0\ne\alpha\in\Lambda}\alpha^\perp$ can be separated from
any other point. 
Thus, the fibers of $\gm_\Lambda^o\to D$ over
all points in the complement of $D\cap\bigcup_{0\ne\alpha\in\Lambda}\alpha^\perp$
consists of just one point.\qqed

\section{Further remarks}\label{sec:GT}

This concluding section explains some consequences of Verbitsky's Global Torelli theorem. Unfortunately,
due to time and space restrictions, I cannot enter a discussion of the polarized case which for an algebraic
geometer is of course the most interesting one. For the latter and in particular for results on
the number of components of moduli spaces of polarized varieties of fixed degree we refer to the relevant sections in
\cite{GHS} and \cite{Mark2}.

\subsection{Bimeromorphic Global Torelli theorem}\label{sec:GT1}

For clarity sake, let us briefly explain again why the generic injectivity of the period
map $\kp:\gm^o_\Lambda \to D\subset\PP(\Lambda \otimes\CC)$ implies that
points in arbitrary fibres are at least birational and how this compares to the Weyl group action for K3 surfaces.
This was alluded to in Remark \ref{rem:passagebim} and corresponds to Proposition \ref{prop:nonsepbirat}.

To illustrate this, let us first go back to the case of K3 surfaces. The moduli space of marked K3 surfaces $\gm$
consists of two connected components interchanged by the involution $(S,\phi)\mapsto (S,-\phi)$.
In particular, for a generic point $x\in D$ the fibre $\kp^{-1}(x)$ consists of exactly two points.
Recall, that the set of generic points $x\in D$ is the complement of a countable union of hyperplane
sections.

The fibre of $\kp:\gm\to D$ over an arbitrary period point $x\in D$ admits a transitive action of
the group of all $\varphi\in {\rm O}(\Lambda)$ fixing $x$. This group is isomorphic to the group
${\rm O}_{\rm Hdg}(H^2(S,\ZZ))$ of all Hodge isometries of $H^2(S,\ZZ)$ of any marked K3 surface $(S,\phi)\in\kp^{-1}(x)$.
The group ${\rm O}_{\rm Hdg}(H^2(S,\ZZ))$ contains the Weyl group $W_S$ generated by all reflections
$s_\delta$ associated to $(-2)$-classes $\delta\in\Pic(S)\cong H^{1,1}(S)\cap H^2(S,\ZZ)$. 
Using the fact that the K\"ahler cone $\kk_S$ of a K3 surface $S$ is cut out from the positive 
cone $\kc_S$ by the hyperplanes $\delta^\perp$, one finds that $\kp^{-1}(x)$ for an arbitrary $x\in D$ 
admits a simply transitive  action of $W_S\times\{\pm1\}$ (see e.g.\ \cite[Exp.\ VII]{BeauvAst}).
Again, $(S,\phi)$ is an arbitrary point in $\kp^{-1}(x)$.

 In particular, the K3 surfaces $S$ and $S'$ underlying two points $(S,\phi),(S',\phi')$  in the same
fibre of $\kp$ are abstractly isomorphic. However, the natural correspondence relating $S$ and $S'$ is not
the graph $\Gamma_g$ of any isomorphism $g:S\cong S'$ but a cycle of the form
$\Gamma:=\Gamma_g+\sum C_k\times C'_k$, where $C_k\subset S$ and $C_k'\subset S'$ are smooth rational curves.
Indeed, if $(S,\phi),(S',\phi')$ are considered as limits of sequences of generic $(S_i,\phi_i)$, resp.\
$(S'_i,\phi'_i)$, with $\kp(S_i,\phi_i)=\kp(S'_i,\phi'_i)$, then the graphs $\Gamma_{g_i}$ of the isomorphisms
$g_i:S_i\congpf S'_i$ deduced from the generic injectivity of $\kp$ (up to sign) will in general not specialize to the
graph of an isomorphism but to a cycle of the form $\Gamma$.

In higher dimensions the situation is similar. Consider two marked compact hyper\-k\"ahler manifolds
$(X,\phi),(X',\phi')$ which are contained in the same connected component $\gm_\Lambda^o$.
Suppose that their periods coincide $\kp(X,\phi)=\kp(X',\phi')$. If the period is generic in $D$,
then Theorem \ref{theo:Verb} proves that $(X,\phi)=(X',\phi')$ as points in $\gm_\Lambda^o$ and thus $X\cong X'$. However, if the period  is not generic, then $X$ and $X'$ might be non-isomorphic. But in this case, they can at least be viewed
as specializations of two sequences $(X_i,\phi_i)$, resp.\ $(X'_i,\phi_i')$, with generic periods
$\kp(X_i,\phi_i)=\kp(X'_i,\phi_i')$ as above. Applying Theorem \ref{theo:Verb} to $(X_i,\phi_i),(X_i',\phi_i')$, shows
the existence of isomorphisms $g_i:X_i\congpf X'_i$ inducing $(X_i,\phi_i)=(X_i',\phi_i')$ as points in $\gm_\Lambda^o$.

Again the graphs $\Gamma_{g_i}$ of the isomorphisms $g_i$ will converge to a cycle $\Gamma\subset X\times X'$,
but $\Gamma$ is more difficult to control. In any case, one can show that $\Gamma$ splits
into $\Gamma=Z+\sum Y_k$ where $X\leftarrow Z\to X'$ defines a bimeromorphic map and none
of the projections $Y_k\to X$ and $Y_k'\to X'$ are dominant (see the proof of Proposition \ref{prop:nonsepbirat}).
This is Theorem 4.3 in \cite{HuyInv} which expresses the fact by saying that non-separated points in $\gm_\Lambda$ are bimeromorphic.

As a consequence of Theorem \ref{theo:Verb}, one can thus state the following

\begin{coro}
Let $(X,\phi),(X',\phi)$ be marked hyperk\"ahler manifolds contained in the same connected component $\gm_\Lambda^o$.
If  $\kp(X,\phi)=\kp(X',\phi')$, then $X$ and $X'$ are bimeromorphic.\qqed
\end{coro}

\subsection{Standard Global Torelli}

Ideally of course, one would like to have a result that deduces from the existence of a Hodge isometry
$H^2(X,\ZZ)\cong H^2(X',\ZZ)$ between two compact hyperk\"ahler manifolds $X$ and $X'$ 
information on the relation between the geometry of the two manifolds. Unfortunately,
Theorem \ref{theo:Verb} fails to produce or to predict such a result. 
As discussed in the introduction, the  generic injectivity of the period map on each connected component 
$\gm_\Lambda^o\subset\gm_\Lambda$, as shown by Verbitsky's  Theorem \ref{theo:Verb}, proves `one half' of the
standard Global Torelli theorem. The `other half', the condition  ${\rm O}(H^2(X,\ZZ))/{\rm Mon}(X)=\{\pm1\}$ on the monodromy action on $H^2(X,\ZZ)$,
does not hold in general. Recall that ${\rm Mon}(X)$
is the subgroup of ${\rm O}(H^2(X,\ZZ))$ generated by the image of all monodromy
representations $\pi_1(B,t)\to{\rm O}(H^2(X,\ZZ))$ induced by smooth proper holomorphic families $\kx\to B$ with
$\kx_t=X$. Here, the base $B$ can be arbitrarily singular.

Thus, in full generality Theorem \ref{theo:Verb} only yields the following weak form of the standard
Global Torelli theorem in which the condition on the monodromy action is not always satisfied and in any
case hard to verify.

\begin{coro} Two compact hyperk\"ahler manifolds $X$ and $X'$ are bimeromorphic
if and only if there exists a Hodge isometry $H^2(X,\ZZ)\cong H^2(X',\ZZ)$ that
can be written as a composition of maps induced by isomorphisms and
parallel transport along paths of complex structures. \qqed
\end{coro}

\begin{coro}\label{cor:MontrueGT}
In particular, if ${\rm O}(H^2(X,\ZZ))/{\rm Mon}(X)=\{\pm1\}$, then 
the bimeromorphic type of $X$ (and for generic $X$ even the isomorphism type)
is determined by its period among compact hyperk\"ahler manifolds that are deformation equivalent to $X$.\qqed
\end{coro}

The monodromy group ${\rm Mon}(X)$ has been computed  by Markman
for $X={\rm Hilb}^n(S)$ and arbitrary $n$ (see \cite{Mark,Mark3}). In particular, his results tell us exactly
when the mono\-dromy condition, and thus the standard Global Torelli theorem for deformations of ${\rm Hilb}^n(S)$,
 do hold. 

\begin{theo}[Markman]
Let $X$ be deformation equivalent to the Hilbert scheme ${\rm Hilb}^n(S)$ of a K3 surface $S$.
Then ${\rm O}(H^2(X,\ZZ))/{\rm Mon}(X)=\{\pm 1\}$ if and only if \mbox{$n=p^k+1$} for some prime number $p$
or $n=1$.\qqed 
\end{theo}

\begin{coro}\label{cor:CorHilb}
Suppose $X$ and $X'$ are deformation equivalent to the Hilbert scheme
${\rm Hilb}^n(S)$ of a K3 surface $S$ such that $n=p^k+1$ for some  prime number $p$.
Then there exists a Hodge isometry $H^2(X,\ZZ)\cong H^2(X',\ZZ)$ if and only if
$X$ and $X'$ are bimeromorphic.\qqed
\end{coro}

\begin{rema} Note that for all other values of $n$ the standard Global Torelli theorem fails, i.e.\ there exist
Hodge isometric deformations of ${\rm Hilb}^n(S)$ that are not bimeromorphic.
  A conjectural explicit description for the monodromy group of the generalized Kummer varieties $K^n(S)$ can be found in
\cite{Mark2}.
\end{rema}

\begin{rema}
Clearly, ${\rm Mon}(X)$ is contained in the image of $\Diff(X)\to{\rm O}(H^2(X,\ZZ))$. However, it
is not known whether the two groups always coincide. For a K3 surface $S$ the computation of
the monodromy group ${\rm Mon}(S)$ is not too difficult. It coincides with the index two subgroup
${\rm O}_+(H^2(S,\ZZ))\subset {\rm O}(H^2(S,\ZZ))$ of all orthogonal transformations preserving
the orientation of a positive three-space. That in this case ${\rm Mon}(S)$  indeed coincides with the action
of the full diffeomorphism group $\Diff(S)$, which is equivalent to the assertion that $-{\rm id}$  is not induced by any
diffeomorphism, is a theorem of Donaldson.
\end{rema}

\subsection{Global Torelli theorem for K3 surface revisited}

As it turns out, Verbitsky's result provides a new approach towards the Global Torelli theorem for K3 surfaces.
Apparently, the potential usefulness of twistor spaces not only for the surjectivity of the period map but also
for its injectivity was discussed among specialists thirty years ago but details have never been worked out. 

We shall briefly explain the situation of K3 surfaces and what precisely is used to prove Theorem
\ref{theo:Verb}.

\smallskip

{\it i) The existence of hyperk\"ahler metrics in each K\"ahler class.} This is a highly non-trivial statement
and uses Yau's solution of the Calabi conjecture. The existence is crucial for Verbitsky's approach as it
ensures the existence of twistor spaces upon which everything else hinges. The theory as represented
in \cite{BeauvAst}, which in turn relies on work of Looijenga and Peters and many others, also uses Yau's result,
but the original proof for algebraic or K\"ahler K3 surfaces due to Pjatecki{\u\i}-{\v{S}}apiro, {\v{S}}afarevi{\v{c}}, 
resp.\  Burns, Rapoport of course did not.

\smallskip

{\it ii) The description of the K\"ahler cone.} More precisely, the proof uses the
fact that a K3 surface $S$ with trivial Picard group has maximal K\"ahler cone, i.e.\
$\kk_S=\kc_S$ (cf.\ Theorem \ref{thm:KK}).  The description of $\kk_S$ for an arbitrary K3 surface $S$ is
much more precise as the statement in higher dimensions (see Remark \ref{rem:KKarb}): $\kk_S$
 is cut out of $\kc_S$ by hyperplanes orthogonal to smooth (!) \!rational curves. (Note that in \cite{BeauvAst} one first proves the surjectivity of the period
map  which is then used to prove this more precise version.)

\smallskip

So far,  the general line of arguments were simply applied to the two-dimensional case.  It would be interesting to see whether
the proofs of {\it i)} and {\it ii)} can be simplified  for K3 surfaces in an essential way. In any case, the arguments to
prove Theorem \ref{theo:Verb} (in arbitrary dimensions) yield that for any connected component $\gm^o$ of the moduli
space of marked K3 surfaces $\gm$ the period map $\kp:\gm^o\to D\subset \PP(\Lambda\otimes\CC)$ is surjective and
generically injective. Moreover, if $(S,\phi),(S',\phi')\in\gm^o$ are contained in the same fibre of $\kp$, then
$S$ and $S'$ are isomorphic.

\smallskip

In order to prove the Global Torelli theorem for K3 surfaces in its original form, one last step is needed (see Corollary \ref{cor:MontrueGT} and page \pageref{pref:GTK3}). (One also needs Kodaira's result that any two K3 surfaces are deformation equivalent. For a rather easy proof see e.g.\ \cite[Exp.\ VI]{BeauvAst}.)

\smallskip
 
{\it iii) For a K3 surface $S$ one has ${\rm O}(H^2(S,\ZZ))/{\rm Mon}(S)=\{\pm1\}$}. Of course, this can be deduced a posteriori
from the Global Torelli theorem for K3 surfaces. But in fact much easier, more direct
arguments exist using classical results on the orthogonal group of unimodular lattices like $2(-E_8)\oplus 2U$
due to Wall, Ebeling, and Kneser.

\smallskip

To conclude, the Global Torelli theorem for K3 surfaces could have been proved along the lines
presented here some thirty years ago. The key step, the properness of the period map $\gm_\Lambda\to D$,
relies on techniques that are very similar to those used for the surjectivity of the period map by Todorov, Looijenga, and Siu.

The main difference of this approach towards the Global Torelli theorem compared to the classical one is that
one does not need to first prove the result for a distinguished class of K3 surfaces, like Kummer surfaces,  and
then use the density of those to extend it  to arbitrary K3 surfaces. Since in higher dimensions no dense distinguished class of hyperk\"ahler manifolds that could play the role of Kummer surfaces is known, this new approach seems the only feasible one.


\end{document}